\documentclass[12pt]{article}
\usepackage{longtable}
\usepackage{textcomp}
\usepackage{amsmath}
\usepackage{amssymb}
\usepackage{mathrsfs}
\usepackage{amsfonts}
\usepackage{multirow}
\usepackage{cases}
\usepackage{graphicx}
\usepackage{tabularx}

\newtheorem{Theorem}{Theorem}[section]
\newtheorem{Definition}{Definition}[section]

\newtheorem{Lemma}{Lemma}[section]

\newtheorem{Corollary}{Corollary}[section]
\newtheorem{algorithm}{Algorithm}

\newcommand{\Section}[1]{
        \par
        \stepcounter{section}
        \settowidth{\hangindent}{\large\bf\thesection.~}
        \hangafter=1
        \bigskip\bigskip\noindent
        {\large\bf\hbox{\thesection.~}#1}\par
        \nopagebreak
        \medskip
}
\newcommand{\alglist}{
\begin{list}{Step 1}
{\setlength{\leftmargin}{1.1 in}\setlength{\labelwidth}{1.0 in}} }

\newenvironment{proof}{{\bf Proof.}\,}{\hfill$\hbox{\rule{5pt}{5pt}}$\\}

\begin{document}
\title{$M$-tensors and  The Positive Definiteness of a Multivariate Form
}

\author{ Liping Zhang$^a$\thanks{The author's  work was supported by the
National Natural Science Foundation of China(Grant No. 10871113).
  \emph{lzhang@math.stinghua.edu.cn}} \quad\quad
Liqun Qi$^b$\thanks{The author's work was supported by the Hong Kong
Research Grant Council. \emph{maqilq@polyu.edu.hk}}\quad\quad
Guanglu Zhou$^c$\thanks{Email address: \emph{G. Zhou@curtin.edu.au}}
\\
{\small $^a$ Department of Mathematical Sciences, Tsinghua University, Beijing, China}\\
{\small $^b$ Department of Applied Mathematics, The Hong Kong
Polytechnic University, Hong Kong}\\
{\small $^c$ Department of Mathematics and Statistics, Curtin
University, Perth, Australia.}
 }

\date{}
\maketitle

{\bf Abstract.} We study $M$-tensors and various properties of $M$-tensors are given. Specially, we show that
the smallest real eigenvalue of $M$-tensor is positive corresponding to a nonnegative eigenvector.
We propose an algorithm to find the smallest positive eigenvalue and
then apply the property to study the positive definiteness of a
multivariate form. \vspace{2mm}

{\bf Key words.} $M$-tensors, eigenvalue, algorithm, positive
definiteness.

{\bf AMS subject classifications. }  65F15, 65F10, 15A18, 15A69

\vspace{2mm}


\Section{Introduction}

$M$-matrices \cite{DingZhou,Horn,Varge} are known to have many applications in modeling dynamic systems in economics,
ecology, and engineering \cite{ela21}. Various properties of $M$-matrices were used in establishing stability results for dynamic systems
in general \cite{ela21,ela22}. Especially, it follows from the famous
Perron-Frobenius theorem for nonnegative matrices that for any
$M$-matrix $A$, all real eigenvalues of $A$ are positive
\cite{Horn,Varge}. The theory of $M$-matrices has many applications
in many fields such as computational mathematics, mathematical
physics, mathematical economics, wireless communications, etc.
\cite{DingZhou,Varge}.  Since avoidance conditions were linked to a stability type of result via Lyapunov-type functions, the theory of $M$-matrices was also used to
certify avoidance conditions \cite{ela}. On the other hand, testing positive definiteness of
a multivariate form is an important problem in the stability study of nonlinear autonomous systems via Lyapunovs direct method in automatic control \cite{LZI18}.
Researchers in automatic control studied the conditions of such positive definiteness intensively  \cite{LZI2, LZI3,LZI5,LZI11}. For $n\ge 3$ and $m\ge 4$, this problem
is a hard problem in mathematics. There are only a few methods to solve the problem \cite{LZI3,LZI5,LZI18}. In practice, when $n>3$ and $m\ge 4$, these methods are computationally expensive. Note that an eigenvalue method was proposed in \cite{LZI18} to solve the problem. In \cite{LZI}, a method for computing the largest eigenvalue of an irreducible nonnegative tensor was applied to test the positive definiteness of a class of of multivariate forms. Motivated by these applications and methods,  we study higher-order $M$-tensors in this paper. We show that all real eigenvalues of an $M$-tensor are positive. Hence, testing the positive definiteness of an even-order multivariate form is equivalent to testing a tensor is an $M$-tensor.

 A tensor can be regarded as a higher-order generalization of a matrix, which
takes the form
\begin{displaymath}
\mathcal{A}=\left(A_{i_1\cdots i_m}\right),\quad A_{i_1\cdots i_m}
\in R,\quad 1\le i_1,\ldots,i_m\le n.
\end{displaymath}
Such a multi-array $\mathcal{A}$ is said to be an {\it$m$-order
$n$-dimensional square real tensor} with $n^m$ entries $A_{i_1\cdots
i_m}$. In this regard, a vector is a first-order tensor and a matrix
is a second-order tensor. Tensors of order more than two are called
higher-order tensors. $M$-tensor we defined in this paper is
ultimately related to the nonnegative tensor. It is a higher order
generalization of the so-called $M$-matrix.

Nonnegative tensors, arising from multilinear pagerank \cite{s2},
spectral hypergraph theory \cite{BP}, and higher-order Markov chains
\cite{s12}, etc., form a singularly important class of tensors and
have attracted more and more attention since they share some
intrinsic properties with those of the nonnegative matrices. One of
those properties is the Perron-Frobenius theorem on eigenvalues. In
\cite{s1}, Chang, Pearson, and Zhang generalized the
Perron-Frobenius theorem from nonnegative matrices to irreducible
nonnegative tensors. Later, Yang and Yang \cite{s11} generalized the
weak Perron-Frobenius theorem to general nonnegative tensors, and
proved that the spectral radius is an eigenvalue of a nonnegative
tensor. Numerical methods for finding the spectral radius of
nonnegative tensors are subsequently proposed. Ng, Qi, and Zhou
\cite{s12} provided an iterative method to find the largest
eigenvalue of an irreducible nonnegative tensor. The Ng-Qi-Zhou
method is efficient but it is not always convergent for irreducible
nonnegative tensors. Zhang and Qi \cite{ZhangQ} established an
explicit linear convergence rate of the Ng-Qi-Zhou method for
essentially positive tensors \cite{pear}.  Liu, Zhou and Ibrahim
\cite{LZI} modified the Ng-Qi-Zhou method such that the modified
algorithm is always convergent for finding the largest eigenvalue of
an irreducible nonnegative tensor. The linear convergence rate of
the algorithm was established in \cite{ZhangQX} for weakly positive
tensors. Chang, Pearson and Zhang \cite{CPZ2} introduced primitive
tensors and established convergence of the Ng-Qi-Zhou method for
primitive tensors.  Friedland, Gaubert and Han \cite{FGH} pointed
out that the Perron-Frobenius theorem for nonnegative tensors has a
very close link with the Perron-Frobenius theorem for homogeneous
monotone maps, initiated by Nussbaum \cite{Nu1, Nu2} and further
studied by Gaubert and Gunawardena \cite{GG}.  Friendland, Gaubert
and Han \cite{FGH} gave a weaker definition for irreducible
nonnegative tensors, and established the Perron-Frobenius theorem in
this context.

By using the eigenvalue theory of nonnegative tensors, we give some
properties of $M$-tensors. We prove that the smallest eigenvalue of
an $M$-tensor is positive with a nonnegative eigenvector. Let
$\mathcal{A}$ be a tensor with nonpositive off-diagonal entries. Two
necessary and sufficient conditions for  $\mathcal{A}$ as an
$M$-tensor are given. Specially, we prove that $\mathcal{A}$ is an
$M$-tensor if and only if its all real eigenvalues are positive. We
also give a sufficient condition for a tensor to be an $M$-tensor,
which is easily verified. Finally, we propose an algorithm for
computing the smallest real eigenvalue. The proposed algorithm is
always convergent for any $M$-tensor. Furthermore, we link
$M$-tensors with a class of multivariate forms and then  apply the
proposed method to study the positive definiteness of a multivariate
form. It should be pointed out that the class of multivariate forms
studied in \cite{LZI} is a special case of our model. We do not need
the assumption that the diagonal entries are positive.

This paper is organized as follows. In Section 2, we will recall
some preliminary results. We  will introduce $M$-tensors and
characterize some basic properties of $M$-tensors in Section 3. In
Section 4, we will propose an iterative algorithm for finding the
smallest real eigenvalue of an $M$-tensor,  and some numerical
results are reported.  In Section 5, we will present an algorithm to
for testing positive definiteness of a class of multivariate forms.
We conclude the paper with some remarks in Section 6.

\Section{Preliminaries}

We start this section with some fundamental notions and properties
on tensors. An $m$-order $n$-dimensional tensor $\mathcal{A}$ is
called nonnegative (or, respectively, positive) if $A_{i_1\cdots
i_m}\ge 0$ (or, respectively, $A_{i_1\cdots i_m}> 0$). The tensor
$\mathcal{A}$ is called symmetric if its entries $A_{i_1\cdots i_m}$
are invariant under any permutation of their indices $\{i_1\cdots
i_m\}$ \cite{s4}. The $m$-order $n$-dimensional unit tensor, denoted
by $\mathcal{I}$, is the tensor whose entries are $\delta_{i_1\ldots
i_m}$ with $\delta_{i_1\ldots i_m}=1$ if and only if $i_1=\cdots
=i_m$ and otherwise zero.  A tensor $\mathcal{A}$ is called
reducible, if there exists a nonempty proper index subset $I\subset
\{1,2,\ldots,n\}$ such that
\begin{displaymath}
A_{i_1\cdots i_m} =0,\quad \forall i_1\in I,\quad \forall
i_2,\ldots,i_m\not\in I.
\end{displaymath}
Otherwise, we say $\mathcal{A}$ is irreducible.

Analogous with that of matrices, the theory of eigenvalues and
eigenvectors is one of the fundamental and essential components in
tensor analysis. Wide range of practical applications can be found
in \cite{s2,s9,s10}. Compared with that of matrices, eigenvalue
problems for higher-order tensors are nonlinear due to their
multilinear structure.
\begin{Definition}\label{eigen}
Let $\textrm{C}$ be the complex field. A pair $(\lambda,x)\in
\textrm{C}\times (\textrm{C}^n\backslash\{0\})$ is called an
eigenvalue-eigenvector pair of $\mathcal{A}$, if they satisfy:
\begin{equation}\label{heigen}
 \mathcal{A}x^{m-1} = \lambda x^{[m-1]},
\end{equation}
where $n$-dimensional column vectors $\mathcal{A}x^{m-1}$ and
$x^{[m-1]}$ are defined as
\begin{eqnarray*}
\mathcal{A}x^{m-1} :=\left(\sum^n_{i_2,\ldots,i_m=1}A_{ii_2\cdots
i_m}x_{i_2}\cdots x_{i_m}\right)_{1\le i\le n}\quad \text{and} \quad
\quad x^{[m-1]}:=\left( x^{m-1}_i\right)_{1\le i\le n},
\end{eqnarray*}
respectively.
\end{Definition}
This definition was introduced in \cite{s09,s2,s4}. We call
$\rho(\mathcal{A})$ the spectral radius of tensor $\mathcal{A}$ if
\begin{displaymath}
\rho(\mathcal{A})=\max\{|\lambda|:\, \text{$\lambda$ is an
eigenvalue of $\mathcal{A}$}\},
\end{displaymath}
where $|\lambda|$ denotes the modulus of $\lambda$. An immediate
consequence on the spectral radius follows directly from Corollary 3
in \cite{s4}.
\begin{Lemma}\label{qi}
Let $\mathcal{A}$ be an $m$-order $n$-dimensional tensor. Suppose
that $\mathcal{B}=a(\mathcal{A}+b\mathcal{I})$, where $a$ and $b$
are two real numbers. Then $\mu$ is an eigenvalue of $\mathcal{B}$
if and only if $\mu=a(\lambda+b)$ and $\lambda$ is an eigenvalue of
$\mathcal{A}$. In this case, they have the same eigenvectors.
\end{Lemma}
Clearly, from this lemma, we have $\rho(\mathcal{B})\le
|a|\left(\rho(\mathcal{A})+|b|\right)$.

Nice properties such as the Perron-Frobenius theorem for eigenvalues of nonnegative square tensors have been established in \cite{s1}.
\begin{Theorem}   \label{thmpf}
If $\mathcal{A}$ is an irreducible nonnegative tensor of order $m$
and dimension $n$, then there exist $\lambda_0 > 0$ and $x_0>0, x_0\in
R^n$ such that
\begin{equation*}
\mathcal{A}x_0^{m-1} = \lambda_0 x^{[m-1]}_0.
\end{equation*}
Moreover, if $\lambda$ is an eigenvalue with a nonnegative
eigenvector, then $\lambda=\lambda_0$. If $\lambda$ is an eigenvalue
of $\mathcal{A}$, then $|\lambda|\le \lambda_0$.
\end{Theorem}
Clearly, by this theorem, $\lambda_0$ is the largest eigenvalue of  $\mathcal{A}$.

 Yang and Yang \cite{s11} asserted that the
spectral radius of a nonnegative tensor is an eigenvalue. In the following we state the results of
\cite[Theorem 2.3 and Lemma 5.8]{s11} for reference.
\begin{Theorem}   \label{YY}
Assume that  $\mathcal{A}$ is a nonnegative tensor of order m
dimension n, then $\rho(\mathcal{A})$ is an eigenvalue of
$\mathcal{A}$ with a nonzero nonnegative eigenvector. Moreover, for
any $x>0, x\in R^n$ we have
\begin{equation*}
\min_{1\le i\le n}\dfrac{(\mathcal{A}x^{m-1})_i}{x^{m-1}_i}\le
\rho(\mathcal{A})\le \max_{1\le i\le
n}\dfrac{(\mathcal{A}x^{m-1})_i}{x^{m-1}_i}.
\end{equation*}
\end{Theorem}
The following continuity of the spectral radius was
given in  the proof of \cite[Theorem
2.3]{s11}.
\begin{Lemma}   \label{YY1}
Let  $\mathcal{A}$ be a nonnegative tensor of order m and dimension
n, and $\varepsilon>0$ be a sufficiently small number. If
$\mathcal{A}_\varepsilon= \mathcal{A}+ \mathcal{E}$ where
$\mathcal{E}$ denotes the tensor with every entry being
$\varepsilon$, then
\begin{equation*}
\lim_{\varepsilon\to
0}\rho(\mathcal{A}_\varepsilon)=\rho(\mathcal{A}).
\end{equation*}
\end{Lemma}

Now we state an iterative algorithm for calculating the largest
eigenvalue of a nonnegative tensor $\cal {A}$, which is proposed in
\cite{LZI, s12} based on Theorems \ref{thmpf} and \ref{YY}.
\begin{algorithm}\label {al31}
\begin{description}
\item[]
\item [{Step 0.}] Choose $x^{(1)} > 0$, $x^{(1)}\in R^n$ and $\rho > 0$.
             Let $\cal {B} =  \cal {A} + \rho
\cal {I}$,  and set $k := 1$.

\item [{Step 1.}] Compute
\begin{eqnarray*}
y^{(k)} &=&  {\cal B}\left(x^{(k)}\right)^{m-1},  \\
\underline{\lambda}_{k} &=& {\rm min}_{x^{(k)}_i > 0}
{\left(y^{(k)}\right)_i
\over \left(x^{(k)}_i\right)^{m-1}}, \\
\bar{\lambda}_{k} &=& {\rm max}_{x^{(k)}_i > 0} {\left(
y^{(k)}\right)_i \over \left(x^{(k)}_i\right)^{m-1}}.
\end{eqnarray*}

\item [{Step 2.}] If $\bar{\lambda}_{k} = \underline{\lambda}_{k}$,
then let $\lambda = \bar{\lambda}_{k}$ and stop. Otherwise, compute
\[
x^{(k+1)} = {{\left(y^{(k)}\right)^{\left[1\over {m-1}\right]}}
\over {\left \|{\left(y^{(k)}\right)^{\left[1\over
{m-1}\right]}}\right\|} },
\]
 replace  $k$ by $k+1$ and go to Step 1.

 \item [{Step 3.}] Output $\lambda - \rho$, the largest
eigenvalue of ${\cal A}$.
\end{description}
\end{algorithm}

It is proved in \cite{LZI} that Algorithm \ref{al31} is always
convergent for irreducible nonnegative tensors; see the following
theorem.

\begin{Theorem}\label {thm4} Suppose ${\cal A}$ be an irreducible
nonnegative tensor.  Let $\cal {B} = \cal {A} + \rho \cal {I}$,
where $\rho > 0$. Assume that $\lambda$ is the largest eigenvalue of
${\cal B}$. Then, Algorithm \ref{al31} produces a value of $\lambda$
in a finite number of steps, or generates two sequences
$\{\underline{\lambda}_{k}\}$ and $\{\bar{\lambda}_{k}\}$ which
converge to  $\lambda$. Furthermore, $\lambda - \rho$ is the largest
eigenvalue of ${\cal A}$.
\end{Theorem}

Note that for any nonnegative tensor $\mathcal{A}$,
$\mathcal{A}+\mathcal{E}$ is an irreducible nonnegative tensor.
Therefore, if we set $\mathcal{B}=\mathcal{A}+\rho \cal
{I}+\mathcal{E}$ in Algorithm \ref{al31}, then by Lemma \ref{YY1},
we can prove that the modified algorithm is also convergent in the
similar way.

\Section{$M$-tensors}
We now extend the notion of $M$-matrices to higher-order tensors and introduce the definition  of an $M$-tensor.
\begin{Definition}\label{defmt}
Let $\mathcal{A}$ be an $m$-order and $n$-dimensional tensor.
$\mathcal{A}$ is called an $M$-tensor if there exist a nonnegative
tensor $\mathcal{B}$ and a real number $c>\rho(\mathcal{B})$, where
$\rho(\mathcal{B})$ is the spectral radius of $\mathcal{B}$, such
that
$$\mathcal{A}=c\mathcal{I}-\mathcal{B}.$$
\end{Definition}
 Clearly, when $m = 2$, if $\mathcal{A}$ is an $M$-tensor then
$\mathcal{A}$ is an $M$-matrix.

Note that the off-diagonal entries of an $M$-tensor are nonpositive.
Denote $\mathbb{Z}$ the set of $m$-order and $n$-dimensional real
tensors whose off-diagonal entries are nonpositive.

 Note that eigenvalues defined in Definition \ref{eigen}
exist for an $M$-tensor \cite{s4}. We will show that  every
eigenvalue of an $M$-tensor has a positive real part, and hence all
of its real eigenvalues are positive.
\begin{Theorem}
Let $\mathcal{A}$ be an $M$-tensor and denote by $\tau(\mathcal{A})$
the minimal value of the real part of all eigenvalues of
$\mathcal{A}$. Then  $\tau(\mathcal{A})>0$ is an eigenvalue of
$\mathcal{A}$ with a nonnegative eigenvector. Moreover, there exist
a nonnegative tensor $\mathcal{B}$ and a real number $c>
\rho(\mathcal{B})$ such that
$\tau(\mathcal{A})=c-\rho(\mathcal{B})$. If $\mathcal{A}$ is
irreducible, then $\tau(\mathcal{A})$ is the unique eigenvalue with
a positive eigenvector.
\end{Theorem}
\begin{proof}
Since $\mathcal{A}$ is an $M$-tensor, by Definition \ref{defmt},  there exist a nonnegative
tensor $\mathcal{B}$ and a real number $c> \rho(\mathcal{B})$
such that
$$\mathcal{A}=c\mathcal{I}-\mathcal{B}.$$
Let $\lambda$ be an eigenvalue of $\mathcal{A}$ and $\mathrm{Re}\lambda$ be
the real part of $\lambda$. By Lemma \ref{qi}, $c-\lambda$ is an eigenvalue of $\mathcal{B}$.
Since $\rho(\mathcal{B})$ is the spectral radius of $\mathcal{B}$,
$$
\rho(\mathcal{B})\ge |c-\lambda|\ge c-\mathrm{Re}\lambda>\rho(\mathcal{B})-\mathrm{Re}\lambda,
$$
which implies that $\mathrm{Re}\lambda>0$, and hence
\begin{equation}\label{eqm1}
\rho(\mathcal{B})\ge \max_{\lambda\in \lambda(\mathcal{A})}\{c-\mathrm{Re}\lambda\}
=c-\min_{\lambda\in \lambda(\mathcal{A})}\{\mathrm{Re}\lambda\}
=c-\tau(\mathcal{A}).
\end{equation}
On the other hand, Theorem \ref{YY} shows that $\rho(\mathcal{B})$
is an eigenvalue of $\mathcal{B}$. By Lemma \ref{qi}, $c-\rho(\mathcal{B})$ is a real eigenvalue
of $\mathcal{A}$, and hence
$$c-\rho(\mathcal{B})\ge \tau(\mathcal{A}),$$
which, together with (\ref{eqm1}), implies
$$\tau(\mathcal{A})=c-\rho(\mathcal{B}).$$
Therefore, $\tau(\mathcal{A})$ is an eigenvalue of $\mathcal{A}$.
Since $\rho(\mathcal{B})$ has a nonnegative eigenvector $x^*$, $x^*$
is also an eigenvector of $\tau(\mathcal{A})$. Moreover, if
$\mathcal{A}$ is irreducible, then $\mathcal{B}$ is also
irreducible. Hence $\rho(\mathcal{B})$ is the unique eigenvalue of
$\mathcal{B}$ with a positive eigenvector.  Thus, we complete the
proof.
\end{proof}

 We state a few of the alternative necessary and
sufficient conditions of a tensor in $\mathbb{Z}$ to be an
$M$-tensor.
 \begin{Theorem}\label{mtensor}
A tensor in $\mathbb{Z}$ is an $M$-tensor if and only if any of its
eigenvalues has a positive real part.
\end{Theorem}
\begin{proof} Let $\mathcal{A}\in \mathbb{Z}$,
and suppose that every eigenvalue of $\mathcal{A}$ has a positive
real part. Let $a=\max\limits_{1\le i\le n}\{A_{i\ldots i}\}$. Then
$\mathcal{B}=a\mathcal{I}-\mathcal{A}$ is nonnegative. By Lemma \ref{qi} and Theorem \ref{YY},
$a-\rho(\mathcal{B})$ is a real eigenvalue of $\mathcal{A}$. It
follows that $a-\rho(\mathcal{B})> 0$. Thus
$\mathcal{A}=a\mathcal{I}-\mathcal{B}$ is an $M$-tensor.

Now, let $\mathcal{A}=c\mathcal{I}-\mathcal{B}$ be an $M$-tensor.
Then, $\mathcal{B}$ is nonnegative and $c> \rho(\mathcal{B})$. Let
$\lambda$ be an eigenvalue of $\mathcal{A}$. Let
$\mathrm{Re}\lambda$ be the real part of $\lambda$. Suppose
$\mathrm{Re}\lambda\le 0$. Since $\lambda$ is an eigenvalue of
$\mathcal{A}$, there exists a nonzero vector $x\in \textrm{C}^n$
such that
 \begin{equation*}
 \mathcal{A}x^{m-1} = \lambda x^{[m-1]},
\end{equation*}
which yields
 \begin{equation*}
 \mathcal{B}x^{m-1} =(c-\lambda) x^{[m-1]}.
\end{equation*}
It follows from Definition 1.1 that $c-\lambda$ is an eigenvalue of
$\mathcal{B}$. Since $c> \rho(\mathcal{B})$ and
$\mathrm{Re}\lambda\le 0$, we have
\begin{equation*}
|c-\lambda|\ge c-\mathrm{Re}\lambda\ge c> \rho(\mathcal{B}),
\end{equation*}
which contradicts the maximality of  $\rho(\mathcal{B})$. We
complete the proof.
 \end{proof}

Another equivalent condition for an $M$-tensor is presented as
follows, which is more easily certified than the one in Theorem
\ref{mtensor}.
 \begin{Theorem}\label{mtensor1}
A tensor $\mathcal{A}\in \mathbb{Z}$ is an $M$-tensor if and only if
all of its  real eigenvalues  are positive.
\end{Theorem}
\begin{proof}
The necessity of the condition follows directly from Theorem
\ref{mtensor}. Suppose now that all real eigenvalues of
$\mathcal{A}\in \mathbb{Z}$ are positive. Let
$\mathcal{B}=a\mathcal{I}-\mathcal{A}$ where $a=\max\limits_{1\le
i\le n}\{A_{i\ldots i}\}$. Then $\mathcal{B}$ is nonnegative and
hence it follows from Theorem \ref{YY} that $\rho(\mathcal{B})$ is
an eigenvalue of $\mathcal{B}$. By Lemma \ref{qi},
$a-\rho(\mathcal{B})$ is a real eigenvalue of $\mathcal{A}$, it must
be positive. That is, $a>\rho(\mathcal{B})$, and therefore
$\mathcal{A}$ is an $M$-tensor.
 \end{proof}

Clearly, we easily obtain  the following results from the above
theorems.
\begin{Corollary}
If $\mathcal{A}\in \mathbb{Z}$ is an $M$-tensor, then
$\max\limits_{1\le i\le n}\{A_{i\ldots i}\}>0$.
\end{Corollary}
\begin{proof} Let $a=\max\limits_{1\le i\le n}\{A_{i\ldots
i}\}$ and $\mathcal{B}=a\mathcal{I}-\mathcal{A}$. Then
$\mathcal{B}\geq 0$ and hence $\rho(\mathcal{B})$ is an eigenvalue
of $\mathcal{B}$. So, $a-\rho(\mathcal{B})$ is a real eigenvalue of
$\mathcal{A}$. Since $\mathcal{A}$ is an $M$-tensor, by Theorem
\ref{mtensor1}, $a-\rho(\mathcal{B})>0$, which implies
$a>\rho(\mathcal{B})\ge 0$.\end{proof}

Note that, If the tensor $\mathcal{A}\in \mathbb{Z}$ then
$\mathcal{A}$ can be deposed into the form
$\mathcal{A}=\lambda\mathcal{I}-\mathcal{B}$ with $\lambda\ge
\max\limits_{1\le i\le n}\{A_{i\ldots i}\}$ and $\mathcal{B}\ge 0$.
Thus, we have the following sufficient and necessary condition for a
tensor being an $M$-tensor, which is a generalization of Theorem 7
in \cite{wood05}.
\begin{Corollary}
Let $\mathcal{A}\in \mathbb{Z}$. Decompose the tensor $\mathcal{A}$
into the form $\mathcal{A}=\lambda\mathcal{I}-\mathcal{B}$, where
$\lambda\ge \max\limits_{1\le i\le n}\{A_{i\ldots i}\}$. Then
$\mathcal{A}$ is an $M$-tensor if and only if
$\lambda>\rho(\mathcal{B})$.
\end{Corollary}
\begin{proof} Clearly, $\mathcal{B}\ge 0$. By Theorem \ref{YY},
$\rho(\mathcal{B}$ is an eigenvalue of $\mathcal{B}$. By Lemma
\ref{qi}, $\lambda-\rho(\mathcal{B})$ is a real eigenvalue of
$\mathcal{A}$. If $\mathcal{A}$ be an $M$-tensor then we have, by
Theorem \ref{mtensor1}, $\lambda-\rho(\mathcal{B})>0$, i.e.,
$\lambda>\rho(\mathcal{B})$. If $\lambda>\rho(\mathcal{B})$ then, by
Definition \ref{defmt}, $\mathcal{A}$ is an $M$-tensor.
\end{proof}

Note that this result provides us an easy method for determining
whether a tensor $\mathcal{A}$ is an $M$-tensor. We only need to
compute the spectral radius of the tensor
$\lambda\mathcal{I}-\mathcal{A}$, where $\lambda\ge
\max\limits_{1\le i\le n}\{A_{i\ldots i}\}$.

Finally, we give a sufficient condition for a tensor to be an
$M$-tensor, which is easily verified. First, we state a definition
which is a generalization from matrices to tensors
\cite{DingZhou,Varge}.
\begin{Definition}
Let $\mathcal{A}$ be an $m$-order and $n$-dimensional tensor.
$\mathcal{A}$ is diagonally dominant if
\begin{equation}\label{diag}
\sum^n_{i_2,\ldots,i_m=1,
\delta_{ii_2\ldots i_m}=0}|A_{ii_2\ldots i_m}|\le
|A_{ii\ldots i}|,\quad i=1,2,\ldots,n,
\end{equation}
where the symbol $\delta_{ii_2\ldots i_m}=0$ is defined as the entry
of the unit tensor $\mathcal{I}$. $\mathcal{A}$ is strictly
diagonally dominant if the strict inequality holds in (\ref{diag})
for all $i$. $\mathcal{A}$ is irreducibly diagonally dominant if
$\mathcal{A}$ is irreducible, diagonally dominant, and the strict
inequality  in (\ref{diag}) holds for at least one $i$.
\end{Definition}

\begin{Theorem}
If a tensor $\mathcal{A}\in \mathbb{Z}$ is strictly or irreducibly
diagonally dominant with all nonnegative diagonal entries, then
$\mathcal{A}$ is an $M$-tensor.
\end{Theorem}
\begin{proof}
Let $\lambda$ be an eigenvalue of $\mathcal{A}$ with a nonzero eigenvector $x$. Denote
$$
|x_i|=\max_{1\le j\le n}|x_j|.
$$
Then
\begin{equation}\label{lag}
\sum_{i_2,\ldots,i_m=1}^nA_{i i_2\ldots i_m}x_{i_2}\cdots x_{i_m}=\lambda x_i^{m-1},
\end{equation}
which implies that
\begin{equation*}
|\lambda-A_{ii\ldots i}|\le \sum^n_{i_2,\ldots,i_m=1;
\delta_{ii_2\ldots i_m}=0}|A_{ii_2\ldots i_m}|.
\end{equation*}
Hence, the diagonal dominance of  $\mathcal{A}$ implies that
\begin{equation}
|\mathrm{Re}\lambda-A_{i\ldots i}|\le |\lambda-A_{i\ldots i}|\le |A_{i\ldots i}|.\label{equ123}
\end{equation}
Since $A_{j\ldots j}\ge
0$ for $j=1,2,\ldots,n$,  (\ref{equ123})
yields
\begin{equation}\label{reeig}
\mathrm{Re}\lambda-A_{i\ldots i}\ge -A_{i\ldots i}.
\end{equation}

Suppose that $\mathcal{A}$ is strictly diagonally dominant. Then the
strict inequality holds in (\ref{diag}) for all $j$, so the strict
inequality holds in (\ref{reeig}). This  yields
$\mathrm{Re}\lambda>0$, i.e., any eigenvalues of $\mathcal{A}$  has
a positive real part. By Theorem \ref{mtensor}, $\mathcal{A}\in
\mathbb{Z}$ is an $M$-tensor.

Suppose now that $\mathcal{A}$ is irreducibly diagonally dominant.
Define
\begin{equation*}
J=\{l:\, |x_l|=\max_{1\le i\le n}|x_i|, \, |x_l|>|x_i| \text{for
some $i$}\}.
\end{equation*}
If $J=\emptyset$, then (\ref{lag}) and the diagonal dominance of
$\mathcal{A}$ imply that for $i=1,2,\ldots,n$,
\begin{equation*}
|\lambda-A_{ii\ldots i}|\le \sum^n_{i_2,\ldots,i_m=1;
\delta_{ii_2\ldots i_m}=0}|A_{ii_2\ldots i_m}|\le |A_{i i\ldots i}|.
\end{equation*}
Let
\begin{equation*}
|A_{kk\ldots k}|> \sum^n_{i_2,\ldots,i_m=1; \delta_{ki_2\ldots
i_m}=0}|A_{ki_2\ldots i_m}|
\end{equation*}
for some $k$. We have
\begin{equation*}
|\mathrm{Re}\lambda-A_{k\ldots k}|\le |\lambda-A_{k\ldots k}|<
|A_{k\ldots k}|=A_{k\ldots k}
\end{equation*}
which implies that $\mathrm{Re}\lambda>0$.

If $J\not=\emptyset$, then the irreducibility  of $\mathcal{A}$
implies that there exist $l\in J$ and $i_2,\ldots,i_m\not\in J$ such
that$ A_{l i_2\ldots i_m}\ne 0$. Hence (\ref{lag}) yields
\begin{equation*}
|\lambda-A_{ll\ldots l}|\le \sum^n_{i_2,\ldots,i_m=1;
\delta_{li_2\ldots i_m}=0}|A_{li_2\ldots
i_m}|\dfrac{|x_{i_2}|}{|x_l|}\cdots \dfrac{|x_{i_m}|}{|x_l|}<
\sum^n_{i_2,\ldots,i_m=1; \delta_{li_2\ldots i_m}=0}|A_{li_2\ldots
i_m}|\le |A_{l l\ldots l}|,
\end{equation*}
which implies that $\mathrm{Re}\lambda>0$. By Theorem \ref{mtensor},
$\mathcal{A}\in \mathbb{Z}$ is an $M$-tensor.
\end{proof}

\Section{An Algorithm}

 Theorem \ref{mtensor1} shows
that a tensor $\mathcal{A}\in \mathbb{Z}$ is an $M$-tensor if and
only if the smallest real eigenvalue of $\mathcal{A}$ is positive.
In this section we propose an algorithm to determine whether or not
a tensor with nonpositive off-diagonal entries is an $M$-tensor.

\begin{Lemma}
Let $\mathcal{A}$ be an $m$-order and $n$-dimensional tensor. Define
\begin{equation}\label{equ41}
L_{\mathcal{A}}=\min_{1\le i\le n}\{A_{i i\ldots i}-C_i\},\quad
U_{\mathcal{A}}=\max_{1\le i\le n}\{A_{i i\ldots i}+C_i\},
\end{equation}
where
\begin{equation*}
C_i=\sum^n_{i_2,\ldots,i_m=1, \delta_{ii_2\ldots
i_m}=0}|A_{ii_2\ldots i_m}|,\quad i=1,2,\ldots,n.
\end{equation*}
Then $L_{\mathcal{A}}$ and $U_{\mathcal{A}}$ are the lower and upper
bounds of real eigenvalues of $\mathcal{A}$, respectively.
\end{Lemma}
\begin{proof}
Let $\lambda$ be a real eigenvalue of $\mathcal{A}$ with an
eigenvector $x\not=0$. That is,
\begin{equation}\label{eq42}
\sum_{i_2,\ldots,i_m=1}^nA_{i i_2\ldots i_m}x_{i_2}\cdots
x_{i_m}=\lambda x_i^{m-1},\quad i=1,2,\ldots,n.
\end{equation}
Let $|x_k|=\max_{1\le i\le n}|x_i|$. Then (\ref{eq42}) implies that
\begin{equation*}
|\lambda-A_{kk\ldots k}|\le \sum^n_{i_2,\ldots,i_m=1,
\delta_{ki_2\ldots i_m}=0}|A_{ki_2\ldots
i_m}|\dfrac{|x_{i_2}|}{|x_k|}\cdots \dfrac{|x_{i_m}|}{|x_k|}\le C_k,
\end{equation*}
which yields $A_{kk\ldots k}-C_k\le \lambda\le A_{kk\ldots k}+C_k$.
This shows $L_{\mathcal{A}}\le \lambda\le U_{\mathcal{A}}$.
\end{proof}

For a tensor $\mathcal{A}\in \mathbb{Z}$, we define a tensor
$\mathcal{C}$ as
\begin{equation}
\mathcal{C}=U_{\mathcal{A}}\mathcal{I}-\mathcal{A},\label{tensorb}
\end{equation}
where $U_{\mathcal{A}}$ is defined in (\ref{equ41}). Clearly, the
tensor $\mathcal{C}$ is a nonnegative tensor. By Lemma \ref{qi} and
Theorem \ref{thmpf}, $U_{\mathcal{A}}-\rho(\mathcal{C})$ is the
smallest eigenvalue of $\mathcal{A}$. By Theorem \ref{mtensor1} and
Definition \ref{defmt}, if $U_{\mathcal{A}}-\rho(\mathcal{C})$ is
positive then $\mathcal{A}$ is an $M$-tensor. Based on this
observation, in the following, we establish an algorithm for
computing the smallest eigenvalue of a tensor $\mathcal{A}\in
\mathbb{Z}$. If the  smallest eigenvalue is positive then
$\mathcal{A}$ is an $M$-tensor. Otherwise, $\mathcal{A}$ is not an
$M$-tensor.

Based on the above discussion, we propose the following algorithm to
determine whether or not a tensor $\mathcal{A}$ with nonpositive
off-diagonal entries is an $M$-tensor.

\begin{algorithm}\label {al41}
\begin{description}
\item[]
\item [{Step 0.}] Compute the upper bound of real eigenvalues,
$U_{\cal A}$ by the formula (\ref{equ41}) and let ${\cal C} =
U_{\cal A}{\cal I} - {\cal A}$.

\item [{Step 1.}] By using Algorithm \ref{al31}, compute $\lambda$,
the largest eigenvalue of ${\cal C}$.

\item [{Step 2.}] Let $\mu = U_{\cal A}  - \lambda$. If $\mu > 0$ then ${\cal A}$ is
an $M$-tensor. Otherwise, ${\cal A}$ is not an $M$-tensor.
\end{description}
\end{algorithm}

In order to show the viability of Algorithm \ref{al41}, we used
Matlab Version 7.7(R2008b) to test it on some tensors with
nonpositive off-diagonal entries which are randomly generated by the
following procedure.

\noindent {\bf Procedure 1. }
\begin{description}
\item{1. }Give $(m, n, A_d)$, where $n$ and $m$ are the dimension and the order of
the randomly generated tensor, respectively, and $A_d > 0$.
\item{2. }Randomly generate an $m$-order $n$-dimensional tensor $\mathcal{D}$
such that all elements of $\mathcal{D}$ are in the interval $(0,
1)$.
\item{3. }Let $\mathcal{A}=\left(A_{i_1\cdots i_m}\right)$,
where $A_{i\cdots i} = A_d + D_{i\cdots i}$, $i = 1, 2, ..., n$,
otherwise, $A_{i_1\cdots i_m} = - D_{i_1\cdots i_m},\quad 1\le
i_1,\ldots,i_m\le n.$
\end{description}

Our numerical results are reported in Table \ref{tab1}  In this
table, {\bf n} and {\bf m} specify the dimension and the order of
the randomly generated tensor, respectively. $A_d$ is a parameter in
Procedure 1.  Given $(m, n, A_d)$, we generate 100 tensors and
determine whether or not they are $M$-tensors by Algorithm
\ref{al41}. In the {\bf Yes} column we show the number of tensors
which are $M$-tensors. In the {\bf No} column, we give the number of
tensors which are not $M$-tensors. CPU(s) denotes the average
computer time in seconds used for Algorithm \ref{al41}. The results
reported in Table \ref{tab1} show that Algorithm \ref{al41} performs
well for these tensors.

\begin{table}
  \begin{center} \footnotesize
    \begin{tabular}{|lll|lll|}  \hline
    $m$ & $n$ & $A_d$ & Yes & No & CPU(s)
    \\  \hline \hline
 3 & 10 & 5  &  0 & 100 & 0.0025 \\
 3 & 10 & 10  &  0 & 100 & 0.0030 \\
 3 & 10 & 100  & 100 &  0 & 0.0028 \\
 3 & 10 & 1000  & 100 &  0 & 0.0028 \\ \hline
 3 & 20 & 5 &  0 & 100 & 0.0094 \\
 3 & 20 & 10  &  0 & 100 & 0.0102 \\
 3 & 20 & 100  &  0 & 100 & 0.0088 \\
 3 & 20 & 1000  & 100 &  0 & 0.0084 \\ \hline
 3 & 30 & 5  &  0 & 100 & 0.0181 \\
 3 & 30 & 10  &  0 & 100 & 0.0187 \\
 3 & 30 & 100  &  0 & 100 & 0.0186 \\
 3 & 30 & 1000  & 100 &  0 & 0.0195 \\ \hline
 3 & 40 & 5  &  0 & 100 & 0.0352 \\
 3 & 40 & 10  &  0 & 100 & 0.0358 \\
 3 & 40 & 100  &  0 & 100 & 0.0362 \\
 3 & 40 & 1000  & 100 &  0 & 0.0355 \\ \hline
 3 & 50 & 5  &  0 & 100 & 0.0619 \\
 3 & 50 & 10  &  0 & 100 & 0.0619 \\
 3 & 50 & 100  &  0 & 100 & 0.0598 \\
 3 & 50 & 1000  &  0 & 100 & 0.0692 \\ \hline
 4 & 10 & 5 &  0 & 100 & 0.0592 \\
 4 & 10 & 10 &  0 & 100 & 0.0603 \\
 4 & 10 & 100  &  0 & 100 & 0.0611 \\
 4 & 10 & 1000  &  0 & 100 & 0.0620 \\ \hline
 4 & 20 & 5 &  0 & 100 & 0.2945 \\
 4 & 20 & 10  &  0 & 100 & 0.3097 \\
 4 & 20 & 100  &  0 & 100 & 0.3187 \\
 4 & 20 & 1000 &  0 & 100 & 0.3116 \\ \hline
 4 & 30 & 5 &  0 & 100 & 1.3233 \\
 4 & 30 & 10 &  0 & 100 & 1.3125 \\
 4 & 30 & 100  &  0 & 100 & 1.3170 \\
 4 & 30 & 1000  &  0 & 100 & 1.3453 \\ \hline
 4 & 40 & 5  &  0 & 100 & 6.5375 \\
 4 & 40 & 10  &  0 & 100 & 6.5358 \\
 4 & 40 & 100 &  0 & 100 & 6.4925 \\
 4 & 40 & 1000  &  0 & 100 & 6.5520 \\ \hline
 4 & 50 & 5  &  0 & 100 & 15.2086 \\
 4 & 50 & 10  &  0 & 100 & 15.1844 \\
 4 & 50 & 100  &  0 & 100 & 15.2102 \\
 4 & 50 & 1000 &  0 & 100 & 15.2039
    \\ \hline
    \end{tabular}
    \caption{ Output of Algorithm \ref{al41} for randomly generated tensors. }
    \label{tab1}
    \end{center}
\end{table}

\Section{An Application: The positive definiteness of a multivariate
form}

In this section we apply the proposed algorithm for testing the
positive definiteness of a class of multivariate forms.

An $m$th degree homogeneous polynomial form of $n$ variables $f(x)$,
where $x\in R^n$, can be denoted as
\begin{equation}
f(x):=\sum_{i_1,i_2,\ldots,i_m=1}^nA_{i_1i_2\ldots
i_m}x_{i_1}x_{i_2}\cdots x_{i_m}.
\end{equation}
When $m$ is even, $f(x)$  is called \emph{positive definite} if
\begin{equation}
f(x)>0,\quad \forall x\in R^n,\, x\not=0.
\end{equation}
Testing positive definiteness of a multivariate form is an important
problem in the stability study of nonlinear autonomous systems
\cite{LZI3,LZI5,LZI18}. It is proved in \cite{s10} that $f(x)$ is
positive definite if and only if its corresponding tensor
$\mathcal{A}=(A_{i_1\ldots i_m})$ is symmetric and all of real
eigenvalues are positive. That is, testing positive definiteness of
a class of even-order multivariate forms is equivalent to determine
whether or not the even-order symmetric tensor is an $M$-tensor. It
should be pointed out that the positive definiteness of a special
class of  multivariate forms was studied in \cite{LZI}. It can be
regarded as our special case. Some numerical results can be referred
\cite{LZI}. We omit them here.

\Section{Conclusions} We have defined and studied $M$-tensors. An
$M$-tensor is the generation of an $M$-matrix. Many important
characterizations of $M$-matrices has been extended to $M$-tensors.
We have proposed two sufficient and necessary conditions, a
necessary condition, and a sufficient condition of an $M$-tensor. We
have shown that the smallest eigenvalue of an $M$-tensor is
positive, and then proposed an algorithm to determine whether or not
a tensor with nonpositive off-diagonal entries is an $M$-tensor.
Finally, we link an $M$-tensor with a multivariate form and apply
the algorithm to judge the positive definiteness of the multivariate
form. Numerical results are reported.

There are some questions are still in study. For example, whether
the condition ``there exists $x\in R^n, x\ge 0$ such that
$\mathcal{A}x^{m-1}>0$" is a necessary and sufficient condition for
a tensor $\mathcal{A}\in \mathbb{Z}$ to be an $M$-tensor?


\begin{thebibliography}{99}
\bibitem{LZI2} N. K. Bose, \emph{Applied Multidimensional System Theory}, New York: Van Nostrand Rheinhold, 1982.

\bibitem{LZI3} N. K. Bose and P. S. Kamat, \emph{Algorithm for stability test of multidimensional filters}, IEEE Trans. Acoust., Speech, Signal Process.,  ASSP-22
(1974) 307-314.

\bibitem{LZI5} N. K. Bose and A. R. Modarressi, \emph{General procedure for multivariable polynomial positivity with control applications},
IEEE Trans. Autom. Control, AC-21 (1976) 696-701.


\bibitem{BP} S.R. Bul\`{o} and M. Pelillo, \emph{New bounds on the clique
number of graphs based on spectral hypergraph theory}, in T.
St\"{u}tzle ed., \emph{Learning and Intelligent Optimization},
Springer Verlag, Berlin, (2009) 45-48.


\bibitem{s1} K. C. Chang, K. Pearson and T. Zhang,  \emph{Perron Frobenius Theorem for nonnegative tensors},
Commu. Math. Sci. 6 (2008) 507-520.

\bibitem{s09} K. C. Chang, K. Pearson and T. Zhang,  \emph{On eigenvalue problems of real symmetric tensors},
J. Math. Anal. Appl. 350 (2009) 416-422.


\bibitem{CPZ2} K.C. Chang, K. Pearson, and T. Zhang,  \emph{Primitivity, the convergence of the NZQ method,
and the largest eigenvalue for nonnegative tensors}, SIAM Journal on
Matrix Analysis and Applications, 32 (2011) 806-819.

\bibitem{DingZhou} J. Ding and A. Zhou, \emph{Nonnegative Matrices, Positive Operators and Applications}, World Scientific Publishing Co. Pte. Ltd., 2009.

\bibitem{FGH} S. Friedland, S. Gaubert and L. Han, \emph{Perron-Frobenius theorem for
nonnegative multilinear forms and extensions}, to appear in: Linear
Algebra and Its Applications.

\bibitem{GG} S. Gaubert and J. Gunawardena, \emph{The
Perron-Frobenius theorem for homogeneous, monotone functions},
Trans. Amer. Math. Soc., 356 (2004), pp. 4931-4950.


\bibitem{LZI11} M. A. Hasan and A. A. Hasan, \emph{A procedure  for the positive definiteness of forms of even-order}, IEEE Trans. Autom. Control, AC-41 (1996) 615-617.

\bibitem{Horn} R. Horn and C. H. Johnson, \emph{Matrix Analysis}, Cambridge University Press, Cambridge, 1996.


\bibitem{s2} L.-H. Lim,  \emph{Singular values and eigenvalues of tensors: a variational
approach}, in Proceedings of the IEEE International Workshop on
Computational Advances in Multi-Sensor Addaptive Processing
(CAMSAP'05), Vol. 1, IEEE Computer Society Press, Piscataway, NJ,
2005, pp. 129-132.


\bibitem{LZI} Y. Liu, G. Zhou, and N.F. Ibrahim, \emph{An always
convergent algorithm for the largest eigenvalue of an irreducible
nonnegative tensor}, J. Compu. Appl. Math. 235 (2010), pp. 286-292.


\bibitem{s12} M. Ng, L. Qi and G. Zhou, \emph{Finding the
largest eigenvalue of a nonnegative tensor}, SIAM J. Matrix Anal.
Appl. 31 (2009) 1090-1099.


\bibitem{LZI18} Q. Ni, L. Qi, and F. Wang, \emph{An eigenvalue method for testing the positive definiteness of a multivariant form}, IEEE Trans. Autom. Control, AC-53 (2008) 1096-1107.


\bibitem{Nu1} R.D. Nussbaum, \emph{Convexity and log convexity for
the spectral radius}, Linear Algebra and Its Applications, 73
(1986), pp. 59-122.

\bibitem{Nu2} R.D. Nussbaum, \emph{Hilbert's projective metric and iterated nonlinear maps}, Memoirs Amer. Math. Soc.,
75 (1988).


\bibitem{pear} K. Pearson, \emph{Essentially positive tensors}, International Journal of
Algebra 4 (2010) 421 - 427.


\bibitem{s4} L. Qi, \emph{Eigenvalues of a real supersymmetric tensor},
J. Symbolic Comput. 40 (2005) 1302-1324.



\bibitem{s9} L. Qi, W. Sun and Y. Wang, \emph{Numerical multilinear
algebra and its applications}, Frontiers Math. China 2 (2007)
 501-526.

\bibitem{s10} L. Qi, Y. Wang and E. X. Wu, \emph{D-eigenvalues of
diffusion kurtosis tensor}, J. Comput. Appl. Math. 221 (2008)
150-157.


\bibitem{ela21} D. D. \v{S}iljak, \emph{Large-scale Dynamic Systems: Stability and Structure}, Noeth-Holland, New York, 1978.

\bibitem{ela22} D. D. \v{S}iljak, \emph{Decentralized Control of Complex Systems}, Academic Press, Boston, 1991.

\bibitem{ela} D. M. Stipanovi\'{c}, S. Shankaran, and C. J. Tomlin, \emph{Multi-agent avoidance control using an $M$-matrix property},
Electronic Journal of Linear Algebra 22 (2005) 64-72.


\bibitem{Varge} R. Varga, \emph{Matrix Iterative Analysis}, Prentice-Hall, Inc., Englewood Cliffs, New Jersey, 1962.

\bibitem{wood05} R. J. Wood and M.J. O'Neill, \emph{A faster
algorithm for identification of an $M$-matrix}, ANZIAM J. 46(E)
(2005) C732-C743.


\bibitem{s11} Y. N. Yang  and Q. Z. Yang, \emph{Further results for Perron-Frobenius Theorem
for nonnegative tensors},  SIAM J. Matrix Anal. Appl. 31 (2010)
2517-2530.

\bibitem{ZhangQ} L. Zhang and L. Qi, \emph{Linear convergence of an algorithm for computing the largest eigenvalue of a nonnegative tensor},
 to appear in: Numerical Linear Algebra with Applications.


\bibitem{ZhangQX} L. Zhang, L. Qi and Y. Xu, \emph{Weakly positive tensors and linear convergence of the LZI algorithm},
Journal of Computation Mathematics, 30 (2012) 24-33.



\end{thebibliography}
\end{document}